\documentclass[a4paper,11pt,fleqn]{amsart}

\usepackage[numbers,sort&compress]{natbib}
\usepackage[pctex32]{graphics}
\usepackage[latin1]{inputenc}
\usepackage{epsfig}
\usepackage[english]{babel}
\usepackage{graphicx,colortbl}
\usepackage{graphics}
\usepackage[dvips]{geometry}
\usepackage{psfrag}
\usepackage{graphicx}
\usepackage[active]{srcltx}
\usepackage{amscd}
\usepackage{amsmath,amstext, amsthm}
\usepackage{amsfonts}
\usepackage{amssymb}
\usepackage{epsfig} 

\usepackage{graphicx,color}
\newtheorem{theorem}{Theorem}[section]

\theoremstyle{definition}

\numberwithin{equation}{section}
\newcommand{\be}{\begin{equation}}
\newcommand{\ee}{\end{equation}}
\newcommand{\C}{\mathbb{C}}

\newcommand{\Q}{\mathbb{Q}}

\newcommand{\kxn}{k[x_1,\dots,x_n]}

\newcommand{\vv}{{\bf V}}

\newcommand{\la}{\langle}

\newcommand{\ra}{\rangle}

\begin{document}

\title{{Isochronicity and linearizability  \\
of a planar cubic system
}}
\author[W. FERNANDES, V.G. ROMANOVSKI, M.S. SULTANOVA, Y. TANG]
{Wilker  Fernandes$^{1}$, Valery G. Romanovski$^{2,3}$, Marzhan  Sultanova$^{4}$, Yilei  Tang$^{2,5}$}

\address{$^1$
Instituto de Ci\^encias Matem\'aticas e de Computa\c{c}\~ao - USP, Avenida Trabalhador S\~ao-carlense, 400, 13566-590, S\~ao Carlos, Brazil}
\email{wilker.thiago@usp.br (W. Fernandes)}

\address{$^2$ Center for Applied Mathematics and Theoretical Physics, University of Maribor, Krekova 2, Maribor,  SI-2000 Maribor, Slovenia}
\email{Valerij.Romanovskij@um.si (V.G. Romanovski)}

\address{$^3$ Faculty of Natural Science and Mathematics, University of Maribor, Koro\v ska c.160, Maribor, SI-2000 Maribor, Slovenia}

\address{$^4$ Faculty of Mechanics and Mathematics, al Farabi Kazakh National University, 71 al-Farabi Ave., Almaty, 050040, Kazakhstan}
\email{marzhan.ss@mail.ru (M. Sultanova)}

\address{$^5$
School of Mathematical Sciences, Shanghai Jiao Tong University, Shanghai, 200240, P.R. China}
\email{Corresponding author. mathtyl@sjtu.edu.cn (Y. Tang)}

\date{}

\begin{abstract}
In this paper we investigate the problem of linearizability for a family of cubic complex planar systems of ordinary differential equations.
 We give a  classification of linearizable systems in the family obtaining  conditions
for linearizability in terms of parameters.
We also discuss  coexistence of isochronous centers in the systems.
\end{abstract}


\maketitle

\section{Introduction}

For planar real analytic differential systems of the form
\begin{equation} \label{sys introd 1}
\dot{x} = -y + P(x,y),
\hspace{0.5cm}
\dot{y} = x + Q(x,y),
\end{equation}
where $P$ and $Q$ are polynomials without constant and linear terms, it is well known that the origin can be either a center or a focus.
In the first case all solutions in a neighbourhood of the origin are periodic  and their trajectories are closed curves.
If the origin is a center, there arises the problem to determine whether all periodic solutions in a neighbourhood of the origin have the same period. This problem is known as the isochronicity problem.

The studies of isochronicity of polynomial differential systems go back to Lloud \cite{Loud}, who found the necessary and sufficient conditions for isochronicity of system \eqref{sys introd 1} when $P$ and $Q$ are homogeneous polynomials of degree two.
Latter on, Pleshkan \cite{P} solved the isochroniciy problem in the case when $P$ and $Q$ are homogeneous polynomials of degree three (see also \cite{Li Jibin}).
In the case when $P$ and $Q$ are homogeneous polynomials of degree five the problem was solved in \cite{R-C-H},
however, the case of homogeneous polynomials of degree four
is still unsolved.
A number of works is devoted to the investigation of some other particular families
(see, e.g \cite{C-G-G-1, C-G-G-2, C-G-G-3, C-S-1,Chen2011, C-G-M-M, H-R-1, Llib2009, Llo1997, M-R-T, R-S, Wu2010}
and references given there).

The following family of planar cubic systems
\begin{equation} \label{system 1}
\begin{aligned}
\dot{x} &= - y + p_2(x,y) +   x r_2(x,y) = P(x,y),
\\
\dot{y} &= x + q_2(x,y)  + y r_2(x,y) = Q(x,y),
  \end{aligned}
\end{equation}
where
\begin{equation*}
\begin{aligned}
p_2 = & a_{20} x^2 + a_{11} x y + a_{02} y^2,
\\
q_2 = & b_{20} x^2 + b_{11} x y + b_{02} y^2,
\\
r_2 = &  r_{20} x^2 + r_{11} x y + r_{02} y^2,
\end{aligned}
\end{equation*}
 has been studied in \cite{C-G, C-G-G, H-R-Z,Llo1997} for the case when all parameters are real.

In \cite{C-G} and \cite{C-G-G} the authors have shown that real  system \eqref{system 1} has a center and an isochronous center, respectively, if and only if
in polar coordinates after some transformations it can be written in one of four and five forms, respectively.
However from their results it is difficult to determine the conditions on parameters of polynomials $p_2, q_2, r_2$ for existence of centers and isochronous centers.
Conditions on parameters of $p_2, q_2, r_2$ for the existence of a center were obtained in \cite{Llo1997}
and latter on using another approach in  \cite{H-R-Z}.

In the work  \cite{Llo1997} published in 1997 the authors
obtained the necessary and sufficient conditions for
existence of isochronous center of system \eqref{system 1}
represented by four series of condition on coefficients
of the system, however in the more recent paper
\cite{C-G-G} published in 1999 the authors gave five conditions for
existence of isochronous center of system \eqref{system 1}.
One of aims  of this paper is to  clarify the conditions
for isochronicity of system \eqref{system 1}. For this purpose
we use the approach different from the ones of \cite{Llo1997}
and \cite{C-G-G}, namely we consider system \eqref{system 1}
as system with complex coefficients and find conditions
for linearization of the system.  We obtain five  series
of conditions for linearizability of \eqref{system 1}
and show that all linearizable systems are Darboux linearizable.

The paper is organized as follows.
In Section \ref{Sec:Prel} we recall some definitions and describe briefly a procedure to study the isochronicity
and linearizability  of  polynomial systems.
Applying this procedure, in Section \ref{Sec:Results} we present our main result, Theorem \ref{Theorem 1},
which gives  conditions for  linearizability of  system \eqref{system 1}.
In Section \ref{Sec:Relation} we present the relation between the results obtained in Theorem \ref{Theorem 1} (and in \cite{Llo1997}) and the results of \cite{C-G-G}.
Finally, in the last section we discuss the
coexistence of isochronous centers in system \eqref{system 1}.

\bigskip

\section{Linearizability quantities and Darboux linearization} \label{Sec:Prel}

In this section we remind some statements related to
  isochronicity and linearizability of
   polynomial differential systems and describe an
   approach to compute the  linearizability quantities
   for the system
   \begin{equation}\label{System general real}
\dot{x} = -y + \sum_{p + q \geq 2}^n a_{p,q} x^{p} y^{q} = P(x,y),
\hspace{0.3cm}
\dot{y} = x + \sum_{p + q \geq 2}^n b_{p,q} x^{p} y^{q} = Q(x,y),
\end{equation}
where $a_{p,q}, b_{p,q} $ are real or complex parameters.

If the equilibrium point at the origin of real  system \eqref{System general real} is known to be a center it is said that this center
is \textit{isochronous}
if all periodic solutions of \eqref{System general real} in a
neighbourhood of the origin have the same period.
System \eqref{System general real} is said to be \textit{linearizable}
if there is an analytic change of coordinates
\begin{equation}\label{lin equa}
x_1 = x +\sum_{m+n \geq 2} c_{m,n} x^m y^n:=H_1(x,y), \hspace{0.5cm} y_1 = y + \sum_{m+n \geq 2} d_{m,n} x^m y^n:=H_2(x,y),
\end{equation}
that reduces \eqref{System general real} to the  linear system
$\dot{x}_1= -y_1$, $\dot{y}_1= x_1$.

The following theorem, which goes back to Poincar\'e and Lyapunov, shows that the linearizability and isochronicity problems are equivalent.
 A proof can be found e.g. in \cite{R-S}.

\begin{theorem}
 The origin of real  system \eqref{System general real} is an
isochronous center if and only if the system is linearizable.
\end{theorem}

Two most common ways to compute isochronicity
quantities (obstacles for isochronicity) are
passing to polar coordinates (the approach used in \cite{C-G-G}) or writing real system \eqref{System general real}
in the complex form
\begin{equation} \label{xZ}
\dot z= i z+Z(z, \bar z),
\end{equation}
using the change $z=x+i y$ and then looking for a linearization
of equation \eqref{xZ} (the approach used in \cite{Llo1997}).
Since we would like to perform the investigation differently
from \cite{C-G-G} and \cite{Llo1997} we use another
computational approach. Namely, we  look for
conditions for linearizability of system
\eqref{System general real} arising from applying
transformation \eqref{lin equa}.

Taking the derivatives  with respect to $t$ on both sides of each equation of \eqref{lin equa} we obtain
\begin{equation*}
\begin{aligned}
\dot{x}_1 = & \dot{x} + \left( \sum_{m+n \geq 2} m c_{m,n} x^{m-1} y^n \right) \dot{x} + \left( \sum_{m+n \geq 2} n c_{m,n} x^m y^{n-1} \right) \dot{y},
\\
\dot{y}_1 = & \dot{y} + \left( \sum_{m+n \geq 2} m d_{m,n} x^{m-1} y^n \right) \dot{x} + \left( \sum_{m+n \geq 2} n d_{m,n} x^m y^{n-1} \right) \dot{y}.
\end{aligned}
\end{equation*}
Hence, the change of coordinates \eqref{lin equa} linearizes system  \eqref{System general real} if it holds that
\begin{equation}\label{diff lin 2}
\begin{aligned}
&\sum_{m+n \geq 2} d_{m,n} x^m y^n + \sum_{p + q \geq 2}^n a_{p,q} x^{p} y^{q} +
\left( \sum_{m+n \geq 2} m c_{m,n} x^{m-1} y^n \right)
\left(-y + \sum_{p + q \geq 2}^n a_{p,q} x^{p} y^{q} \right)
\\
& + \left( \sum_{m+n \geq 2} n c_{m,n} x^m y^{n-1} \right)
 \left(x + \sum_{p + q \geq 2}^n b_{p,q} x^{p} y^{q} \right)  \equiv 0,
\\
&  - \sum_{m+n \geq 2} c_{m,n} x^m y^n + \sum_{p + q \geq 2}^n b_{p,q} x^{p} y^{q} +
\left( \sum_{m+n \geq 2} m d_{m,n} x^{m-1} y^n \right)
\left(-y + \sum_{p + q \geq 2}^n a_{p,q} x^{p} y^{q} \right)
\\
& + \left( \sum_{m+n \geq 2} n d_{m,n} x^m y^{n-1} \right)
 \left(x + \sum_{p + q \geq 2}^n b_{p,q} x^{p} y^{q} \right)  \equiv 0.
 \end{aligned}
\end{equation}
Obstacles for the fulfilment of equations in \eqref{diff lin 2} give us  necessary conditions for  existence of a linearizing change of coordinates \eqref{lin equa} of system \eqref{System general real}.
Thus,  a computational procedure to find necessary  conditions for linearizability can be described as follows.

(1) Write the left hand sides of two equations in \eqref{diff lin 2} in the form
$\sum_{k,l \geq 2} h_1^{(k,l)} x^k y^l$, and $\sum_{k,l \geq 2} h_2^{(k,l)} x^k y^l$,
respectively, where $h_1^{(k,l)}$ and $h_2^{(k,l)}$  are polynomials in the parameters $a_{p,q}, b_{p,q}$  ($p+q \geq 2$)
of system \eqref{System general real}  and $c_{m,n}$, $d_{m,n}$ ($m+n \geq 2$) of \eqref{lin equa}.

(2) Solve the polynomial system $h_i^{(k,l)} = 0$  ($i=1,2$,  $k+l=2$) for the coefficients  $c_{m,n}$, $d_{m,n}$ ($m+n=2$) of \eqref{lin equa}.

 (3) Solve the polynomial system  $h_i^{(k,l)} = 0$ ($i=1,2$, $k+l=3$)  for the coefficients  $c_{m,n}$, $d_{m,n}$ ($m+n=3$) of \eqref{lin equa}.
In general case the system cannot be solved.
However dropping from it two suitable equations  we obtain a system that has a solution.
We denote the two dropped polynomials on the left hand sides of these two
equations  by $i_1$ and $j_1$.

(4) Proceed step-by-step solving the polynomial systems  $h_i^{(k,l)} = 0$ ($i=1,2$, $k+l=r$, $r>3$).
Generally speaking, at all steps when $r=k+l$ is an odd number  the polynomial system  $h_i^{(k,l)} = 0$ ($i=1,2$, $k+l=r$) cannot be solved.
Dropping on each such step  two suitable equations (and  denoting by $i_{(r-1)/2}$ and $j_{(r-1)/2}$ the corresponding polynomials), we obtain a system that has a solution.

This procedure yields the polynomials $i_k$ and $j_k$ which are  polynomials in the parameters $a_{p,q}$ and $b_{p,q}$ of system \eqref{System general real} called the \textit{linearizability quantities}.
It is clear that  system  \eqref{System general real} admits a linearizing change of coordinates \eqref{lin equa} if and only if $i_k = j_k =0$ for all $k>1$.
Thus, the simultaneous vanishing of all linearizability quantities provide conditions which characterize when the system  \eqref{System general real} is linearizable (equivalently
it has an isochronous center at the origin).
The ideal $\mathcal{L} = \langle i_1, j_1, i_2, j_2, ...  \rangle \subset \mathbb{C}[a,b]$  defined by the linearizability quantities is called the \textit{linearizability ideal}
and its affine variety, $V_{\mathcal{L}} = {\bf V} (\mathcal{L})$,  is called the \textit{linearizability variety}.
Therefore, the  linearizability problem will be solved finding the variety $V_{\mathcal{L}}$.

By the Hilbert Basis Theorem there exists a positive integer $k$ such that  $\mathcal{L} =
\mathcal{L}_{k} = \langle i_1, j_1,... ,i_k, j_k \rangle$.
Note that the inclusion $V_{\mathcal{L}} \subset {\bf V} (\mathcal{L}_k)$ holds for any $k \geq 1$.
The opposite inclusion is verified finding the irreducible decomposition of the variety ${\bf V}( \mathcal{L}_k )$ and then checking that any point of each component of the decomposition corresponds to a linearizable system.
The irreducible decomposition can be found  using the routine \texttt{minAssGTZ} \cite{D-L-P-S} (which is based on the algorithm of \cite{GTZ}) of the computer algebra system {\sc Singular} \cite{sing}, however it involves extremely laborious
calculations.


One of the most efficient method to find a linearizing change of coordinates  is the Darboux linearization method.
To construct a Darboux linearization it is convenient to
perform the  substitution
\be \label{zw}
z=x+iy, \qquad w=x-iy
\ee
 obtaining from \eqref{System general real} a system of the form
\begin{equation}  \label{System general complex}
\dot{z} = i (z + X(z,w )),  \qquad \dot{w}= - i ( w + Y(z,w) ),
\end{equation}
and, after the rescaling of time by $i$, the system
\begin{equation} \label{System general complex-1}
\dot{z}= z + X(z,w) , \qquad \dot{w}= - w - Y(z,w).
\end{equation}
Since the change of coordinates \eqref{zw} is analytic,   system \eqref{System general real}
is linearizable if and only if system  \eqref{System general complex-1}
is linearizable.

We remind that a \textit{Darboux factor} of system \eqref{System general complex-1} is a polynomial $f(z,w)$ satisfying
\begin{equation*}
\dfrac{\partial f}{\partial z} \dot{z} + \dfrac{\partial f}{\partial w} \dot{w} = K f,
\end{equation*}
where $K(z,w)$ is a polynomial called the \textit{cofactor of $f$}.
A \textit{Darboux linearization} of system \eqref{System general complex-1} is an analytic change of coordinates  $z_1 =  Z_1 (z,w)$,
$w_1 = W_1 (z,w)$, such that
\begin{equation*}
\begin{aligned}
Z_1(z,w) = \prod_{j=0}^{m} f_j^{\alpha_j}(z,w) = z + \tilde{Z}_1 (z,w),
~~
W_1(z,w) = \prod_{j=0}^{n} g_j^{\beta_j}(z,w) = w + \tilde{W}_1 (z,w),
\end{aligned}
\end{equation*}
which linearizes \eqref{System general complex-1},
where $f_j, g_j \in \mathbb{C}[z,w]$, $\alpha_j, \beta_j \in
\mathbb{C}$, and $\tilde{Z}_1$ and $\tilde{W}_1$ have neither constant  nor
linear terms.
System \eqref{System general complex-1} is said to be Darboux linearizable if it admits a Darboux linearization.
The next theorem provides a way to construct a Darboux linearization using Darboux factors (see e.g. \cite{R-S} for a proof).

\begin{theorem}
\label{Theorem Darboux lineari}
System \eqref{System general complex-1} is Darboux linearizable if and only if there exit $s+1 \geq 1$ Darboux factors $f_0,...,f_s$ with corresponding cofactors $K_0,...,K_s$ and $t+1 \geq 1$ Darboux factors $g_0,...,g_t$ with corresponding cofactors $L_0,...,L_t$  with the following properties:
\begin{enumerate}
	\item[a.] $f_0(z,w) = z +  \cdot \cdot \cdot \mbox{ but }  f_j(0,0) = 1 \mbox{ for } j \geq 1$;
	\item[b.] $g_0(z,w) = w +  \cdot \cdot \cdot \mbox{ but }   g_j(0,0) = 1 \mbox{ for }  j \geq 1$; and
	\item[c.] there are $s+t$ constants $\alpha_1,...,\alpha_s, \beta_1,...,\beta_t \in \mathbb{C}$ such that
\begin{equation}  \label{cond 1}
K_0 + \alpha_1 K_1 + \cdot \cdot \cdot + \alpha_s K_s = 1
\hspace{0.3cm} \mbox{  and  } \hspace{0.3cm}
L_0 + \beta_1 L_1 + \cdot \cdot \cdot + \beta_t L_t = -1.
\end{equation}
\end{enumerate}

The Darboux linearization is then given by
\begin{equation*}
z_1 = H_1(z,w) = f_0 f_1^{\alpha_1} \cdot \cdot \cdot f_s^{\alpha_s},
\qquad
y_1 = H_2(z,w) = g_0 g_1^{\beta_1} \cdot \cdot \cdot g_t^{\beta_t} .
\end{equation*}
\end{theorem}

Sometimes we cannot find enough Darboux factors to construct  Darboux linearizations of both equations of the system.
Let say that we can find only transformation $z_1$, which linearizes the first equation of \eqref{System general complex-1}.
If we can find a first integral of system \eqref{System general complex-1} of the form $\Psi=xy+ h.o.t.$
then the second equation of \eqref{System general complex-1} can be linearized by the transformation $w_1 = \dfrac{\Psi}{z_1}$.
We note also that   if  system \eqref{System general complex-1}  has $p$ irreducible
Darboux factors $f_1, ... , f_p$ with associated cofactors $K_1, ... , K_p$, satisfying
$s_1K_1+...+s_pK_p = 0$,     then  the  function $ H = f_1^{s_1}... f_p^{s_p}$
 is a first integral of \eqref{System general complex-1}.


\section{Linearizability of system \eqref{system 1}} \label{Sec:Results}

In this section we obtain the necessary and sufficient conditions for linearizability of system \eqref{system 1}
with complex parameters.

Without loss of generality, we suppose that $b_{02} = - b_{20}$  in system \eqref{system 1}. Otherwise, if $a_{02} +a_{20}=0$, we can apply
the transformation $\tilde{x}=x+(b_{02}+ b_{20})y/(a_{02} +a_{20})$ and  $\tilde{y}=y-(b_{02}+ b_{20})x/(a_{02} +a_{20})$.
When $a_{02} +a_{20}=0$, we only need to make the change $(x,y)\to (y,x)$ together with the time scaling
 $dt=-d\tau$ and   obtain the same effect.

The following theorem gives the conditions for linearizability of system  \eqref{system 1}.

\begin{theorem} \label{Theorem 1}
Complex system \eqref{system 1} with $b_{02} = - b_{20}$ is
linearizable at the origin if one of the following conditions holds:
\begin{enumerate}
	\item[(1)]  $	4 a_{20}^2 + a_{11}^2 + 4 a_{11}  b_{20} + 4 b_{20}^2 - 4 a_{20} b_{11} + b_{11}^2 =
	r_{20}+r_{02}=a_{02}+a_{20}=0$,
	\item[(2)] $  a_{02} = r_{02} = a_{11} + 2 b_{20} = b_{11} - 4 a_{20} =  r_{11} + b_{20}^2 = r_{20} -
 a_{20} b_{20} = 0$,
	\item[(3)] $ 4 a_{02} + a_{20} = a_{11} + 2 b_{20} = 2 b_{11} - a_{20} = 4 r_{02} + a_{20} b_{20} =  r_{11} + b_{20}^2 = r_{20} - a_{20} b_{20} = 0$,
 	\item[(4)] $  a_{02} = r_{02} = a_{11} + 2 b_{20} =  b_{11} - a_{20} =  r_{20} - a_{20} b_{20} = 0$,
 \item[(5)] $9 a_{11}^2-12 a_{11} b_{20}+4 b_{20}^2+4 b_{11}^2=
 -6 a_{11} b_{20}+4 b_{20}^2+2 a_{20} b_{11}-b_{11}^2 =
 6 a_{20} a_{11}-4 a_{20} b_{20}-3 a_{11} b_{11}+10 b_{20} b_{11} =
  4 a_{20}^2-12 a_{11} b_{20}+24 b_{20}^2-b_{11}^2 =
  - \frac{4}{3}  b_{20}^2 -  \frac{b_{11}^2}{3} + r_{11}=
   \frac{4}{9} a_{20} b_{20} +  \frac{a_{11} b_{11}}{6} -  \frac{b_{20} b_{11}}{9} + r_{02}=
    \frac{a_{20} a_{11}}{6} -  \frac{a_{20} b_{20}}{3} +  \frac{a_{11} b_{11}}{12} - \frac{b_{20} b_{11}}{6} + r_{20}+r_{02}=
    a_{02}+ \frac{a_{20}}{3} - \frac{b_{11}}{3} = 0$.
\end{enumerate}
\end{theorem}

\begin{proof}
Using the computer algebra system {\sc Mathematica} following the computational procedure described in the previous section  we computed the first eight pairs of the linearizability quantities for system \eqref{system 1}.
The first pair is
\begin{equation*}
\begin{aligned}
i_1 = & \frac{1}{9} (10 a_{02}^2 + a_{11}^2 + 10 a_{02} a_{20} + 4 a_{20}^2 -
   a_{02} b_{11} - 5 a_{20} b_{11} + b_{11}^2 + 4 a_{11} b_{20} + 4 b_{20}^2 ),
   \\
j_1 = & \frac{1}{3} ( a_{02} a_{11} + a_{11} a_{20} - 2 a_{02} b_{20} - 2 a_{20} b_{20}  +  4 r_{02} + 4 r_{20} ),
\end{aligned}
\end{equation*}
and the second pair reduced by the Groebner basis of $ \la i_1,j_1 \ra $ is
\begin{equation*}
\begin{aligned}
\tilde i_2 = & \frac{1}{750} (-10 a_{11}^2 a_{20}^2 + 200 a02 a_{20}^3 + 160 a_{20}^4 +
   10 a_{11}^2 a_{20} b_{11} - 600 a_{02} a_{20}^2 b_{11} - 520 a_{20}^3 b_{11}
   \\
   & +
   6 a_{11}^2 b_{11}^2 + 490 a_{02} a_{20} b_{11}^2 + 464 a_{20}^2 b_{11}^2 -
   96 a_{02} b_{11}^3 - 110 a_{20} b_{11}^3 + 6 b_{11}^4 - 170 a_{11}^3 b_{20}
   \\
   & -
   720 a_{11} a_{20}^2 b_{20} + 720 a_{11} a_{20} b_{11} b_{20} - 146 a_{11} b_{11}^2 b_{20} -
   550 a_{11}^2 b_{20}^2 + 2600 a_{02} a_{20} b_{20}^2
   \\
   & + 3080 a_{20}^2 b_{20}^2 -
   1580 a_{02} b_{11} b_{20}^2 - 2060 a_{20} b_{11} b_{20}^2 + 154 b_{11}^2 b_{20}^2 -
   160 a_{11} b_{20}^3 + 520 b_{20}^4
   \\
   & - 100 a_{11} a_{20} r_{02} + 560 a_{11} b_{11} r_{02} -
   6800 a_{20} b_{20} r_{02} + 2180 b_{11} b_{20} r_{02} + 5250 r_{02}^2
   \\
   & - 55 a_{11}^2 r_{11} +
   50 a_{02} a_{20} r_{11} - 170 a_{20}^2 r_{11} + 5 a_{02} b_{11} r_{11} + 225 a_{20} b_{11} r_{11} -
   55 b_{11}^2 r_{11}
   \\
   & - 220 a_{11} b_{20} r_{11} - 220 b_{20}^2 r_{11} - 100 a_{11} a_{20} r_{20} +
   560 a_{11} b_{11} r_{20} + 800 a_{02} b_{20} r_{20}
   \\
   & - 6000 a_{20} b_{20} r_{20} +
   2180 b_{11} b_{20} r_{20} + 8500 r_{02} r_{20} + 3250 r_{20}^2),
\\
\tilde j_2 = & \frac{1}{120} (2 a_{11}^3 a_{20} + 8 a_{11} a_{20}^3 - a_{11}^3 b_{11} - 12 a_{11} a_{20}^2 b_{11} +
   6 a_{11} a_{20} b_{11}^2 - a_{11} b_{11}^3 - 4 a_{11}^2 a_{20} b_{20}
   \\
   & + 48 a_{02} a_{20}^2 b_{20} -
   6 a_{11}^2 b_{11} b_{20} + 16 a_{02} a_{20} b_{11} b_{20} + 56 a_{20}^2 b_{11} b_{20} -
   4 a_{02} b_{11}^2 b_{20}
   \\
   & - 8 a_{20} b_{11}^2 b_{20} - 2 b_{11}^3 b_{20} -
   40 a_{11} a_{20} b_{20}^2 - 12 a_{11} b_{11} b_{20}^2 - 48 a_{20} b_{20}^3 - 8 b_{11} b_{20}^3
   \\
   & - 24 a_{11}^2 r_{02} - 16 a_{20}^2 r_{02} + 64 a_{20} b_{11} r_{02} + 4 b_{11}^2 r_{02} -
   64 a_{11} b_{20} r_{02} - 32 b_{20}^2 r_{02}
   \\
   & + 128 a_{02} b_{20} r_{11} +
   128 a_{20} b_{20} r_{11} - 128 r_{02} r_{11} - 8 a_{11}^2 r_{20} - 32 a_{02} a_{20} r_{20} +
   16 a_{20}^2 r_{20}
   \\
   & - 80 a_{02} b_{11} r_{20} - 80 a_{20} b_{11} r_{20} + 20 b_{11}^2 r_{20} +
   32 b_{20}^2 r_{20} - 128 r_{11} r_{20}).
\end{aligned}
\end{equation*}
The other polynomials have very long expressions, so we do not present  them here, however, the reader can easily compute them using any
available computer algebra system.

To find conditions for linearizability we have to ``solve"
the system $ i_1=\dots =i_8=j_1=\dots = j_8=0$,
or, more precisely, to find the irreducible
decomposition of the variety
${\bf V} (\mathcal{L}_8)$ of the ideal  $\mathcal{L}_8= \langle i_1, j_1,...,i_8, j_8 \rangle $.
Although  nowadays there are few algorithms for computing
such decompositions the  calculations seldom can be completed
over the field of rational numbers for non-trivial
ideals due to high complexity of Groebner bases computations.
We tried to perform the  decomposition of the variety of
${\bf V} (\mathcal{L}_8)$  using the routine
  \texttt{minAssGTZ} \cite{D-L-P-S}  of  {\sc Singular} \cite{sing},
  however we have not succeeded  to complete computations
  neither  over $\mathbb{Q}$ nor over the  field $\mathbb{Z}_{32003}$.

  To find the decomposition we   proceed as follows. First, we use
  the  conditions for isochronicity of real system \eqref{system 1} obtained in \cite{Llo1997},
  which are conditions (2)--(4) of the statement of the  theorem   and the condition
\be \label{cond1*}
  a_{02}  + a_{20} =  a_{11} + 2 b_{20} = b_{11} - 2 a_{20} = r_{02} +  r_{20}  = 0.
\ee
It is clear that under condition \eqref{cond1*} and conditions (2)--(4) of the theorem 
complex system \eqref{system 1} is linearizable as well.

Denote by $J_1$ the ideal generated by polynomials defining condition \eqref{cond1*},
that is,
$$
J_1=\la  a_{02}  + a_{20} ,  a_{11} + 2 b_{20} , b_{11} - 2 a_{20} , r_{02} +  r_{20}\ra,
$$
and by $J_2, J_3, J_4$ ideals generated by polynomials of conditions (2)--(4) of the theorem.

 As we have mentioned above
 we are not able to compute the decomposition
 of the variety ${\bf V} (\mathcal{L}_8)$ of $\mathcal{L}_8$ (that is, to find
 the minimal associate primes of  $\mathcal{L}_8$)
 even over fields of finite characteristic.
 However using the ideals $ J_1-J_4$
 we can  find the decomposition
 of the variety ${\bf V} (\mathcal{L}_8)$.
 The idea is to  subtract from ${\bf V} (\mathcal{L}_8)$ the
 components defined by the ideals  $ J_1-J_4$
 and then  find the decomposition of the remaining variety.
 For this aim we use the theorem (see, e.g.
 \cite[Chapter 4]{Cox}) for the proof),
 which says that given two ideals
  $I$ and $H$  in $\kxn$,
\be\label{geincd}
\overline{\vv(I) \setminus \vv(H)} \subset \vv(I : H)\,,
\ee
where the overline indicates the Zariski closure. Moreover, if $k = \C$ and $I$ is a radical
ideal, then
\be\label{geincd2-IH}
\overline{\vv(I) \setminus \vv(H)} = \vv(I : H).
\ee
Thus, to remove the components ${\bf V}(J_1), ..., {\bf V}(J_4)$ from ${\bf V}(\mathcal{L}_8)$,
we  compute  over the field  $\mathbb{Z}_{32003}$  with the  \texttt{intersect} of Singular
 the intersection
  $ J=  J_1 \cap  J_2 \cap  J_3 \cap {J}_4 $ (clearly,   $ \vv(J) = \vv(\ J_1) \cup \vv(\ J_2) \cup \vv( J_3) \cup \vv({J}_4)$),
 then with the  \texttt{radical} we compute $R=\sqrt{\mathcal{L}_8}$, then  with  \texttt{quotient}
  we compute the ideal  $ G= R:J  $ and, finally, with \texttt{minAssGTZ}
  we compute the minimal associate primes of $G$,
  obtaining that $G=G_1\cap G_2$, where \\
 $G_1=\la
  r_{20}+r_{02}
, a_{02}+a_{20}
a_{20}^2+8001 a_{11}^2+a_{11} b_{20}+b_{20}^2-a_{20} b_{11}+8001 b_{11}^2\ra$\\
and \\
$
G_2= \la a_{02}+10668 a_{20}-10668 b_{11},
 a_{11} r_{02}-10667 b_{20} r_{02}+14224 a_{20} r_{11}-14224 b_{11} r_{11},
a_{20} r_{02}+16000 b_{11} r_{02}+16001 a_{11} r_{11}-b_{20} r_{11},
r_{20}^2+6 r_{20} r_{02}+9 r_{02}^2+r_{11}^2,
b_{11} r_{20}+3 b_{11} r_{02}-16000 a_{11} r_{11}-b_{20} r_{11},
b_{20} r_{20}+3 b_{20} r_{02}-16001 a_{20} r_{11}-8001 b_{11} r_{11},
a_{11} r_{20}-2 b_{20} r_{02}-a_{20} r_{11}-16001 b_{11} r_{11},
a_{20} r_{20}-15997 b_{11} r_{02}+8003 a_{11} r_{11}-16001 b_{20} r_{11},
a_{11} b_{11}-2 b_{20} b_{11}+4 r_{20}+6 r_{02},
b_{20}^2+8001 b_{11}^2+8000 r_{11},
\\
a_{11} b_{20}-10668 a_{20} b_{11}+10668 b_{11}^2+16001 r_{11},
a_{20} b_{20}-16001 b_{20} b_{11}+16000 r_{20},
a_{11}^2-14224 a_{20} b_{11}-7111 b_{11}^2-10668 r_{11},
a_{20} a_{11}+b_{20} b_{11}+r_{20}+3 r_{02},
a_{20}^2-a_{20} b_{11}+8000 b_{11}^2+3 r_{11},
a_{20} b_{11} r_{11}-16001 b_{11}^2 r_{11}-9 r_{20} r_{02}-27 r_{02}^2-6 r_{11}^2,
b_{20} b_{11} r_{02}-8001 b_{11}^2 r_{11}+8003 r_{02}^2+r_{11}^2,
a_{20} b_{11}^2-16001 b_{11}^3-18 b_{20} r_{02}-6 a_{20} r_{11},
b_{11}^2 r_{02} r_{11}+b_{20} b_{11} r_{11}^2+15997 r_{20} r_{02}^2+15988 r_{02}^3-r_{20} r_{11}^2+15997 r_{02} r_{11}^2,
b_{11}^3 r_{02}+b_{20} b_{11}^2 r_{11}-9 b_{20} r_{02}^2-3 b_{11} r_{02} r_{11}-4 b_{20} r_{11}^2
\ra.
$ \\
Since $8001 \equiv \frac 14 {\ \rm mod \ } 32003$, lifting the ideal
$G_1$ from the ring of polynomials over  the field $\mathbb{Z}_{32003}$ to the ring of polynomials over  the field  $\Q$ we obtain
polynomials given in  condition (1) of the theorem.

Similarly, lifting the ideal $G_2$   we obtain the ideal which we denote by $J_5$
(the lifting can be performed algorithmically using the algorithm of \cite{WGD} and the {\sc Mathematica} code of \cite{GCPRS}).
Simple computations show that $\vv(J_5)$ is
 the same set  as the set  given by conditions (5) of the theorem.

To check the correctness of the obtained
conditions we use the procedure described  in \cite{R-P}. 
First, we computed the ideal
$\tilde J=  J_1 \cap \ J_2 \cap  J_3 \cap {J}_4 \cap {J}_5 $, which defines the union
of all five  components of the theorem and
   checked  that Groebner bases of all
ideals $\la \tilde{J}, 1-w i_k \ra$,    $\la \tilde{J}, 1-w j_k \ra$ (where $k=1,\dots, 8$ and $w$ is a new variable) computed over $\mathbb{Q}$
are $\{1\}$. By the Radical Membership Test (see e.g. \cite{Cox, R-S}) it means that
\be \label{invc2}
\vv(\mathcal{L}_8)\subset \vv (\tilde{J}).
\ee
  To check the opposite inclusion
it is sufficient to check that
 \be  \label{clm}
 \la \mathcal{L}_8, 1-w f \ra =\la 1 \ra
 \ee
for  all polynomials $f$ from a basis
of     $ \tilde{J}$. Unfortunately,  we were not able
to perform the check over $\mathbb{Q}$,
however we have checked that \eqref{clm}
holds over few fields of finite characteristic.
It yields that \eqref{clm} holds with high probability \cite{EA}
\footnote{For this reason we say in the statement of Theorem \ref{Theorem 1} that conditions (1)-(5) are only necessary, but not
necessary and sufficient conditions for linearizability of system \eqref{system 1}.}.

We now prove that under each of conditions
(1)-(5) of the theorem the system is linearizable.

\medskip

$Case \hspace{0.1cm} (1):$ In this case
 $a_{11} = -2 b_{20} \pm  (2 a_{20} - b_{11})i $.
We consider only the case
 $a_{11} = -2 b_{20} +  (2 a_{20} - b_{11})i $,
 since
 when   $a_{11} = -2 b_{20} - (2 a_{20} - b_{11})i $
 the consideration is analogous.
In this case system \eqref{system 1} becomes
\begin{equation} \label{case a real}
\begin{aligned}
\dot{x}= &  - y + a_{20} x^2 +(-2 b_{20} + (2 a_{20} - b_{11})  i) x y  - a_{20} y^2 + r_{20} x^3  + r_{11} x^2 y -  r_{20} x y^2,
\\
\dot{y}= & x + b_{20} x^2 + b_{11} x y - b_{20} y^2 + r_{20} x^2 y  + r_{11} x y^2 -  r_{20} y^3.
\end{aligned}
\end{equation}
After the   substitution \eqref{zw} we obtain from   \eqref{case a real} the system
\begin{equation} \label{case a complex}
\begin{aligned}
\dot{z}= & z-     (i a_{20} -   b_{20}) z^2   -\frac{1}{4}( r_{11} + 2 i r_{20}) z^3 +\frac{1}{4}   (r_{11} -2 i r_{20}) z w^2,
\\
\dot{w}= &  -w +\frac{1}{2} ( i b_{11}-2 i a_{20}) z^2  - \frac{1}{2}(   i b_{11} + 2 b_{20}) w^2
-\frac{1}{4} (r_{11} + 2 i r_{20}) z^2 w + \frac{1}{4}(r_{11}   - 2 i r_{20}) w^3.
\end{aligned}
\end{equation}
System \eqref{case a complex}
has  Darboux factors
\begin{equation*}
\begin{aligned}
l_1 =& z, 
 \\
l_3 =&   1 + \frac{1}{16}    ( -8 i a_{20} + 8 b_{20} -      4 \sqrt{2}       \eta_-  ) z
        ~  +  \frac{1}{4}  (i b_{11} + 2 b_{20} -     i \xi) w,
 \\
l_4 =& 1 +
  \frac{1}{16}  (-8 i a_{20} + 8 b_{20} +
     4 \sqrt{2}       \eta_-) z ~
       +  \frac{1}{4}  (i b_{11} + 2 b_{20} - i \xi) w;,
\\
l_5 =&   1 +  \frac{1}{16} (-8 i a_{20} + 8 b_{20} -
     4 \sqrt{2}        \eta_+) z
       ~+ \frac{1}{4}(i b_{11} + 2 b_{20} + i \xi) w,
\\
l_6 =&   1 + \frac{1}{16}  (-8 i a_{20} + 8 b_{20} +
     4 \sqrt{2}       \eta_+) z +
  \frac{1}{4} (i b_{11} + 2 b_{20} + i \xi) w,
 \end{aligned}
\end{equation*}
where $ \xi = \sqrt{b_{11}^2 - 4 i b_{11} b_{20} - 4 b_{20}^2 - 4 r_{11} + 8 i r_{20}}$ and
\[
\eta_{\pm}= \sqrt{-2 a_{20}^2 + 2 a_{20} b_{11} - b_{11}^2 - 8 i a_{20} b_{20} +
       2 i b_{11} b_{20} + 2 b_{20}^2 + 2 r_{11}
        + 4 i r_{20} \pm 2 a_{20} \xi \mp b_{11} \xi }.
\]
It is easy to verify that the first of  conditions \eqref{cond 1} is satisfied with $f_0 = l_1$,  $f_1 = l_4$, $f_2 = l_5$,  $f_3 = l_6$,  and
 \begin{equation*}
 \begin{aligned}
\alpha_1 =&    - \frac{b_{11} - 2 i b_{20} + \xi}{2 \xi},
\\
\alpha_2 =&
\frac{ b_{11} \eta_+ - 2 i  b_{20} \eta_+ - b_{11} \eta_- + 2 i  b_{20} \eta_- - 2 i  \sqrt{2} a_{20} \xi +
 2 \sqrt{2} b_{20} \xi - \eta_+ \xi - \eta_- \xi}{4 \eta_+ \xi},
\\
\alpha_3 =&  \frac{b_{11} (\eta_+ + \eta_-) + (2 i  \sqrt{2} a_{20} - \eta_+ + \eta_-) \xi -
 2 i  b_{20} (\eta_+ + \eta_- - i  \sqrt{2} \xi)}{4 \eta_+ \xi}.
 \end{aligned}
\end{equation*}
Moreover, the system has the Darboux first integral
\[
\Psi(z,w)=l_3^{s_1} l_4^{s_2} l_5^{s_3} l_6^{s_4}=1 - \frac{i}{2\sqrt{2}} \eta_-\xi zw + o(||(z,w)||^3),
\]
where
$
s_1 =  1, ~ s_2 =  -1, ~
s_3 = - \frac{\eta_-}{\eta_+}, ~
s_4 =   \frac{\eta_-}{\eta_+},
$
   $f_1 = l_3$, $f_2 = l_4$, $f_3 = l_5$, and $f_4 = l_6$. 

Therefore, the system  is  linearizable 
 by the substitution
 \begin{equation*}
 z_1 =  l_1 l_4^{\alpha_1} l_5^{\alpha_2} l_6^{\alpha_3},
\hspace{0.3cm}
 w_1 =\frac{2\sqrt{2}(\Psi(z,w)-1)i}{\eta_-\xi z_1}.
 \end{equation*}

\medskip

$Case \hspace{0.1cm} (2):$ In this case
system \eqref{system 1} becomes
\begin{equation} \label{case b-1 real}
\begin{aligned}
\dot{x}= &  - y + a_{20} x^2 + a_{20} b_{20} x^3 - 2 b_{20} x y - b_{20}^2 x^2 y=   (b_{20} x+1) (a_{20} x^2-b_{20} x y-y),
\\
\dot{y}= & x + b_{20} x^2 + 4 a_{20} x y + a_{20} b_{20} x^2 y - b_{20} y^2 - b_{20}^2 x y^2,
\end{aligned}
\end{equation}
and
after  substitution \eqref{zw} we have the system
\begin{equation} \label{case b-1 complex}
\begin{aligned}
\dot{z}= & z + \left(b_{20} - i \frac{5}{4} a_{20} \right)  z^2
- \frac{i}{2} a_{20} z w
+ i \frac{3}{4} a_{20} w^2
+ \left( \frac{b_{20}^2}{4}  - \frac{i}{4} a_{20} b_{20} \right) z^3
\\
& - \frac{i}{2} a_{20} b_{20} z^2 w
- \left( \frac{ b_{20}^2}{4} +  \frac{i}{4} a_{20} b_{20} \right) z w^2,
\\
\dot{w}= & - w + i \frac{3}{4} a_{20} z^2
- \frac{i}{2} a_{20} z w
- \left( b_{20} + i \frac{5}{4} a_{20} \right) w^2
+ \left(  \frac{b_{20}^2}{4} - \frac{i}{4} a_{20} b_{20} \right) z^2 w
\\
& - \frac{i}{2} a_{20} b_{20} z w^2
- \left( \frac{b_{20}^2}{4} + \frac{i}{4} a_{20} b_{20}  \right) w^3.
\end{aligned}
\end{equation}
System \eqref{case b-1 complex} has the Darboux factors
\begin{equation*}
\begin{aligned}
 l_1 =& z + \left(\frac{b_{20}}{2} + \frac{i}{4} a_{20} \right) z^2
  + \left(\frac{b_{20}}{2} + \frac{i}{2} a_{20} \right) z w
  + \frac{i}{4}a_{20} w^2,
 \\
 l_2 =& w - \frac{i}{4} a_{20} z^2
 + \left( \frac{b_{20}}{2} - \frac{i}{2} a_{20} \right) z w
 + \left(\frac{b_{20}}{2} - \frac{i}{4} a_{20} \right) w^2,
 \\
l_3 =&   1 + \frac{b_{20}}{2} z + \frac{b_{20}}{2} w,
\\
l_4 =& 1 - \frac{i}{2} \left(4 a_{20} + i b_{20} \right) z + \frac{1}{2} \left( b_{20} + 4 i a_{20} \right) w,
 \end{aligned}
\end{equation*}
which  when $a_{20} \neq 0$  allow to construct the Darboux linearization
 \begin{equation}
 \label{AC2}
 z_1 =  l_1 l_3^{\alpha_1} l_4^{\alpha_2},
\hspace{0.3cm}
 w_1 =   l_2 l_3^{\beta_1} l_4^{\beta_2},
 \end{equation}
where
 \begin{equation*}
 \begin{aligned}
\alpha_1 =& - \frac{6 a_{20} - i b_{20}}{4 a_{20}},
\hspace{0.3cm}
\alpha_2 = - \frac{2 a_{20} + i b_{20}}{4 a_{20}},
\\
 \beta_1 =& - \frac{6 a_{20} + i b_{20}}{4 a_{20}},
 \hspace{0.3cm}
\beta_2 = - \frac{2 a_{20} - i b_{20}}{4 a_{20}}.
 \end{aligned}
\end{equation*}

%

Since the set of linearizable system is an affine variety and therefore it is a closed set in the Zariski topology,
the system is linearizable  also when $a_{20} = 0$.

\medskip

$Case \hspace{0.1cm} (3):$ In this case system \eqref{system 1} becomes
\begin{equation} \label{case b-2 real}
\begin{aligned}
\dot{x}= &  - y + a_{20} x^2 - 2 b_{20} x y - \frac{a_{20}}{4} y^2
 + x \left( a_{20} b_{20} x^2 - b_{20}^2 x y - \frac{ a_{20} b_{20}}{4}  y^2 \right),
\\
\dot{y}= & x + b_{20} x^2 + \frac{a_{20}}{2} x y - b_{20} y^2
+ y \left( a_{20} b_{20} x^2 - b_{20}^2 x y - \frac{a_{20} b_{20}}{4} y^2 \right),
\end{aligned}
\end{equation}
and 
the corresponding system of the form \eqref{System general complex-1} is
\begin{equation} \label{case b-2 complex}
\begin{aligned}
\dot{z}= & z
+ \left( b_{20} - i \frac{7}{16} a_{20} \right) z^2
- i \frac{3}{8} a_{20} z w
- i \frac{3}{16} a_{20} w^2
+ \left(\frac{b_{20}^2}{4} - i \frac{5}{16} a_{20} b_{20} \right) z^3
\\
& - i \frac{3}{8} a_{20} b_{20} z^2 w
- \left( \frac{b_{20}^2}{4} +  i \frac{5}{16} a_{20} b_{20} \right) z w^2,
\\
\dot{w}= & - w
 - i \frac{3}{16} a_{20} z^2
 - i \frac{3}{8} a_{20} z w
  - \left( b_{20} + i \frac{7}{16} a_{20} \right)  w^2
 + \left( \frac{b_{20}^2}{4} - i \frac{5}{16} a_{20} b_{20} \right)  z^2 w
 \\
& - i \frac{3}{8} a_{20} b_{20} z w^2
 - \left(\frac{b_{20}^2}{4} + i \frac{5}{16} a_{20} b_{20} \right) w^3 .
\end{aligned}
\end{equation}
System \eqref{case b-2 complex} has the following Darboux factors
\begin{equation*}
\begin{aligned}
 l_1 =& z
+ \left( \frac{b_{20}}{2}  - \frac{i}{16} a_{20} \right) z^2
+ \left( \frac{b_{20}}{2} + \frac{i}{8} a_{20} \right) z w
-  \frac{i}{16} a_{20} w^2,
 \\
 l_2 =& w
 + \frac{i}{16} a_{20} z^2
 + \left( \frac{b_{20}}{2} - \frac{i}{8} a_{20} \right) z w
 + \left( \frac{b_{20}}{2} + \frac{i}{16} a_{20} \right) w^2,
 \\
l_3 =&   1 + \frac{b_{20}}{2} z + \frac{b_{20}}{2} w,
\\
l_4 =&1 - \frac{i}{4} \left( a_{20} + i 2  b_{20} \right) z
 + \frac{i}{4} \left( a_{20} - i 2 b_{20} \right) w,
 \end{aligned}
\end{equation*}
which when $a_{20} \neq 0$ allow to construct the Darboux linearization
 \begin{equation}
  \label{AC3}
 z_1 =  l_1 l_3^{\alpha_1} l_4^{\alpha_2},
\hspace{0.3cm}
 w_1 =   l_2 l_3^{\beta_1} l_4^{\beta_2},
 \end{equation}
 where
 \begin{equation*}
 \begin{aligned}
\alpha_1 =& \frac{i 2  b_{20}}{a_{20}},
\hspace{0.3cm}
\alpha_2 = - \frac{2 a_{20} + i 2 b_{20}}{a_{20}},
\\
 \beta_1 =& - \frac{i 2 b_{20}}{a_{20}},
 \hspace{0.3cm}
\beta_2 = - \frac{2 a_{20} - i 2 b_{20}}{a_{20}}.
 \end{aligned}
\end{equation*}

 If $a_{20} = 0$,  case (3) is equivalent to case (2).
Thus system \eqref{case b-2 complex} is linearizable.

\medskip

$Case \hspace{0.1cm} (4):$
In this case system \eqref{system 1} becomes
\begin{equation} \label{case b-4 real}
\begin{aligned}
\dot{x}= &  - y + a_{20} x^2  - 2 b_{20} x y + a_{20} b_{20} x^3  + r_{11} x^2 y,
\\
\dot{y}= & x + b_{20} x^2 + a_{20} x y  - b_{20} y^2 + a_{20} b_{20} x^2 y + r_{11} x y^2,
\end{aligned}
\end{equation}
and after the substitution \eqref{zw} we obtain the system
\begin{equation} \label{case b-4 complex}
\begin{aligned}
\dot{z}= &  z
+ \left( b_{20} - i/2 a_{20} \right)  z^2
- \frac{i}{2} a_{20} zw
- \left(\frac{ r_{11} }{4} + \frac{i}{4} a_{20} b_{20} \right) z^3
- \frac{i}{2} a_{20} b_{20} z^2 w
\\
& ~~+ \left(\frac{r_{11}}{4} - \frac{i}{4} a_{20} b_{20} \right) zw^2,
\\
\dot{w}= & -w
- \frac{i}{2} a_{20} zw
- \left( b_{20} + \frac{i}{2} a_{20} \right) w^2
- \left(\frac{ r_{11}}{4} + \frac{i}{4} a_{20} b_{20} \right) z^2w
- \frac{i}{2} a_{20} b_{20} z w^2
\\
&~~+ \left(\frac{r_{11}}{4}  - \frac{i}{4} a_{20} b_{20} \right) w^3,
\end{aligned}
\end{equation}
which admits the Darboux factors
\begin{equation*}
\begin{aligned}
 l_1 =& z, \hspace{0.3cm} l_2 = w,
 \\
 l_3 =&  1 + \frac{1}{4} \left( - i a_{20} + 2 b_{20} + i C \right) z
 - \frac{i}{4} \left( -a_{20} + i 2  b_{20} + C \right) w,
 \\
l_4 =& 1 - \frac{i}{2} \left( a_{20} + i 2 b_{20} \right) z
+  \frac{i}{2} \left( a_{20} - i 2 i b_{20} \right) w
- \frac{i}{4} \left( a_{20} b_{20} - i r_{11} \right) z^2
\\
& + \frac{1}{2} \left( 2 b_{20}^2 + r_{11} \right) z w
+  \frac{i}{4} \left( a_{20} b_{20} + i r_{11} \right) w^2,
 \end{aligned}
\end{equation*}
where $C = \sqrt{a_{20}^2 - 4 b_{20}^2 - 4 r_{11}}$.
When $C \neq 0$ we obtain the Darboux linearization
 \begin{equation}
 \label{AC4}
 z_1 =  l_1 l_3^{\alpha_1} l_4^{\alpha_2},
\hspace{0.3cm}
 w_1 =   l_2 l_3^{\beta_1} l_4^{\beta_2},
 \end{equation}
where
 \begin{equation*}
 \begin{aligned}
\alpha_1 =& \frac{a_{20} + i 2  b_{20}}{C},
\hspace{0.3cm}
\alpha_2 = - \frac{a_{20} + i 2 b_{20} + C}{ 2 C},
\\
 \beta_1 =& \frac{a_{20} - i 2 b_{20}}{C},
 \hspace{0.3cm}
\beta_2 = - \frac{a_{20} - i 2 b_{20} + C}{ 2 C}.
 \end{aligned}
\end{equation*}

Using the same argument as in case (2) we conclude that the system is linearizable also when $C=0$.

\medskip

$Case \hspace{0.1cm} (5):$  If $b_{20} \ne 0$, 
we can rewrite the condition as
\begin{equation*}
\begin{aligned}
&r_{11} = 3 a_{02}^2 + 2 a_{20} a_{02} + \frac{a_{20}^2}{3} +  \frac{4b_{20}^2}{3},
~~r_{02} =  \frac{27 a_{02}^3 + 9 a_{02}^2 a_{20} - 3 a_{02} a_{20}^2 - a_{20}^3 - 16 a_{20} b_{20}^2}{36 b_{20}},
\\
&  a_{11} = - \frac{9 a_{02}^2 - a_{20}^2 - 4 b_{20}^2}{6b_{20}}, ~r_{20} = a_{02} b_{20} + a_{20} b_{20},
~b_{11} = a_{20} + 3 a_{02},  ~a_{20} = 3 a_{02} \pm 4 b_{20} i.
\end{aligned}
\end{equation*}
We only consider the case $ a_{20} = 3 a_{02} + 4 b_{20} i $, 
since when  $a_{20} = 3 a_{02} - 4 b_{20} i$, the consideration is analogous. Under this condition after the substitution \eqref{zw} system \eqref{system 1} becomes
\begin{equation} \label{case-v1}
\begin{aligned}
\dot{z}= &   z + (3 b_{20}- 3 i a_{02}) z^2  + (2 b_{20}- 2 i a_{02}) z w+ 2 i a_{02} w^2  +(2 b_{20}^2- 2 a_{02}^2  - 4 i a_{02} b_{20} ) z^3
\\
& - (2 a_{02}^2 + 4 i a_{02} b_{20}- 2 b_{20}^2) z^2 w
+ (4 a_{02}^2   + 4 i a_{02} b_{20}) z w^2,
\\
\dot{w}= &  -w \Big(1 + (2 i a_{02}- 2 b_{20}) z + (i a_{02}  - b_{20}) w + (2 a_{02}^2+ 4 i a_{02} b_{20}- 2 b_{20}^2) z^2
\\
& + (2 a_{02}^2 + 4 i a_{02} b_{20}-2 b_{20}^2) z w - (4 a_{02}^2 + 4 i a_{02} b_{20}) w^2 \Big).
\end{aligned}
\end{equation}

System \eqref{case b-1 complex} has the  Darboux factors
\begin{equation*}
\begin{aligned}
 l_1 =&  ~z - i (a_{02} + i b_{20}) z^2 +\frac{ 2 i a_{02}}{3} w^2,
 \\
 l_2 =&   ~w,
 \\
l_3 =&   ~  1 - 2 i (a_{02} + i b_{20}) z +
  i (a_{02} + i b_{20}) w,
\\
l_4 =&  ~
1 - 4 i (a_{02} + i b_{20}) z - 4 (a_{02} + i b_{20})^2 z^2 +
  2 i (a_{02} + i b_{20}) w +
  4 (a_{02} + i b_{20})^2 z w
  \\
  &  + (-a_{02}^2 - 2 i a_{02} b_{20} +
     b_{20}^2) w^2, \end{aligned}
\end{equation*}
which allow to construct the Darboux linearization
 \begin{equation}
 \label{AC211}
 z_1 =  l_1 l_4^{-1},
\hspace{0.3cm}
 w_1 =   l_2 l_4^{-\frac{1}{2}}.
 \end{equation}

Similarly as above, using the Zariski closure argument we conclude that the system is linearizable also when $b_{20}=0$.
\end{proof}

\bigskip


\section{Relation between isochronicity conditions of \cite{C-G-G} and Theorem \ref{Theorem 1}} \label{Sec:Relation}

In \cite{C-G-G} the authors presented conditions for isochronicity of system \eqref{system 1} when all parameters of the system are real.
We investigate the relation between their  conditions and the conditions presented in  Theorem \ref{Theorem 1} and in \cite{Llo1997}.
The following result is obtained in \cite{C-G-G}.

\begin{theorem} [Theorem 1 of \cite{C-G-G}]
\label{th-e}
The origin of system \eqref{system 1} is an isochronous center if and only if \eqref{system 1} can be transformed in one of the following
forms in polar coordinates:
\begin{enumerate}
	\item[(a)]
$
\begin{cases}
 \dot{r} = r^2 (\cos3\theta-\frac{7}{3}\cos\theta-k_1 \sin\theta)
+r^3(-\frac{2k_1}{3}-\frac{2k_1}{3} \cos2\theta-\frac{k_1^2}{2} \sin2\theta),
\\
\dot{\theta} = 1+r (-\sin3\theta+k_1\cos \theta- \sin\theta),
\end{cases}
$
	\item[(b)]
$
\begin{cases}
 \dot{r} = r^2 (\cos3\theta+\frac{13}{3}\cos\theta-k_1 \sin\theta)
+r^3(2k_1+\frac{10k_1}{3} \cos2\theta-\frac{k_1^2}{2} \sin2\theta),
\\
\dot{\theta} = 1+r (-\sin3\theta+k_1\cos \theta+ \frac{1}{3}\sin\theta),
\end{cases}
$
\item[(c)]
$
\begin{cases}
 \dot{r} = r^2 k_1\cos\theta
+r^3( k_2 \cos2\theta+k_3 \sin2\theta),
\\
\dot{\theta} = 1+r k_1\cos \theta,
\end{cases}
$
\item[(d)]
$
\begin{cases}
 \dot{r} = r^2 (k_1\cos\theta+k_2 \sin\theta)
+r^3(\frac{k_1k_2}{2}-\frac{k_1k_2}{2} \cos2\theta+k_3 \sin2\theta),
\\
\dot{\theta} = 1+r k_1 \sin\theta
\end{cases}
$
 and
\item[(e)]
$
\begin{cases}
 \dot{r} = r^2 (k_1\cos\theta+k_2 \sin\theta)
+r^3(k_3+k_4 \cos2\theta+k_5 \sin2\theta),
\\
\dot{\theta} = 1,
\end{cases}
$
\end{enumerate}
where $k_j$'s in each system are independent and are functions of original parameters in system \eqref{system 1}.
\end{theorem}


As it is mentioned in the previous section by the result of \cite{Llo1997}  real system \eqref{system 1}
is linearizable (equivalently, it has isochronous center)
if and only if  condition \eqref{cond1*} or one of conditions (2)-(4) of
 Theorem  \ref{Theorem 1} hold.
The following theorem gives the relation of the results  of \cite{Llo1997} (and Theorem \ref{Theorem 1})
and \cite{C-G-G}.

\begin{theorem}
\label{th-e2}
System \eqref{system 1}  under conditions \eqref{cond1*}, $(2)$, $(3)$, and $(4)$ of Theorem \ref{Theorem 1}  can be changed
into system  (c), (a), (b) and (d)  of  Theorem \ref{th-e},  respectively.
\end{theorem}

\begin{proof}
 System \eqref{system 1} under condition \eqref{cond1*}  becomes
 \begin{equation} \label{sys1-1}
\begin{aligned}
\dot{x} &= - y + a_{20} x^2 - 2b_{20} x y - a_{20} y^2   +   x(r_{20} x^2 + r_{11} x y -r_{20} y^2)   = P_1(x,y),
\\
\dot{y} &= x +   b_{20} x^2 + 2a_{20} x y   -b_{20} y^2 + y(r_{20} x^2 + r_{11} x y -r_{20} y^2)   = Q_1(x,y).
  \end{aligned}
\end{equation}

Applying the linear transformation
\begin{equation*}
x= -\frac{ a_{20}}{a_{20}^2 + b_{20}^2}  \tilde{x}+  \frac{b_{20}}{a_{20}^2+ b_{20}^2}  \tilde{y} ,
~~ y= \frac{b_{20}}{a_{20}^2+  b_{20}^2} \tilde{x}+  \frac{a_{20}}{a_{20}^2+ b_{20}^2}  \tilde{y},
\end{equation*}
and a time scaling $dt= - d\tilde{t}$,
we change system  \eqref{sys1-1} to
\begin{equation}
\label{3-e}
\begin{aligned}
 \dot{x} & = -y + x^2 - y^2 + x \left(k_2 x^2 + 2 k_3 xy -  k_2 y^2 \right),
\\
\dot{y}& =  x +2 xy  +y \left(k_2 x^2 + 2 k_3 xy -k_2 y^2 \right),
\end{aligned}
\end{equation}
where
$k_2 = \frac{a_{20} b_{20} r_{11} + a_{20}^2 r_{20} + b_{20}^2 r_{20}}{\left( a_{20}^2 + b_{20}^2 \right)^2}$,
$k_3 = \frac{a_{20}^2 r_{11} - b_{20}^2 r_{11}  + 4 a_{20} b_{20} r_{20}}{2 \left(a_{20}^2 + b_{20}^2 \right)^2}$,
and below we  write $x$ and $y$ instead of $\tilde{x}$ and $\tilde{y}$.
System \eqref{3-e} in polar coordinates $x = r \cos \theta$, $y = r \sin \theta$ becomes system (c).

System \eqref{system 1} under condition (2) becomes system \eqref{case b-1 real}.
The transformation
$
x= \frac{4}{3 a_{20}}  \tilde{x} ,
~~ y= -\frac{4}{3 a_{20}}  \tilde{y}
$
and the time scaling $dt=-d\tilde{t}$ change system \eqref{case b-1 real} to
\begin{equation} \label{sys case1 trans}
\begin{aligned}
\dot{x} = & -y - \frac{4}{3} x^2 - 2 k_1 x y - \frac{x}{3} \left(4 k_1 x^2 + 3 k_1^2 x y \right),
\\
\dot{y} = & x + k_1 x^2 - \dfrac{16}{3} x y - k_1 y^2 - \frac{y}{3} \left(4 k_1 x^2 + 3 k_1^2 x y \right),
\end{aligned}
\end{equation}
where we  write $x$ and $y$ instead of $\tilde{x}$ and $\tilde{y}$,  and $k_1 = \frac{4 b_{20}}{3 a_{20}}$.
System \eqref{sys case1 trans} in polar coordinates $x = r \cos \theta$, $y = r \sin \theta$ becomes system (a).

System \eqref{system 1} under condition (3) becomes system \eqref{case b-2 real}.
Applying the transformation
$
x= \frac{16}{3 a_{20}}  \tilde{x} ,
~~ y= \frac{16}{3 a_{20}}  \tilde{y},
$
we transform  \eqref{case b-2 real}  to the system
\begin{equation} \label{sys case2 trans}
\begin{aligned}
\dot{x} = & -y - \frac{16}{3} x^2 - 2 k_1 x y - \frac{4}{3} y^2 + \frac{k_1}{3} x \left(16 x^2 - 3 k_1 x y - 4 y^2 \right),
\\
\dot{y} = & x + k_1 x^2 + \dfrac{8}{3} x y - k_1 y^2 + \frac{k_1}{3} y \left(16 x^2 - 3 k_1 x y - 4 y^2 \right),
\end{aligned}
\end{equation}
where we  write $x$ and $y$ instead of $\tilde{x}$ and $\tilde{y}$, and $k_1 = \frac{16 b_{20}}{3 a_{20}}$.
System \eqref{sys case2 trans} in polar coordinates $x = r \cos \theta$, $y = r \sin \theta$ becomes system (b).

 System \eqref{system 1} under condition (4) becomes system \eqref{case b-4 real}.
The transformation
$x=  \tilde{y}$, $y= \tilde{x}$  and a time scaling $dt=-d\tilde{t}$ change
system \eqref{case b-4 real} to
\begin{equation} \label{sys case5 trans}
\begin{aligned}
\dot{x} = & -y + k_1 x^2 +k_2 x y -k_1 y^2 + x \left( 2 k_3 x y +   k_1 k_2  y^2 \right),
\\
\dot{y} = & x + 2 k_1 x y + k_2 y^2 + y \left(  2 k_3 x y +  k_1 k_2  y^2 \right),
\end{aligned}
\end{equation}
where $k_1 = b_{20}$, $k_2 = -a_{20}$, $k_3 = -\frac{r_{11} }{2}$, and  we write $x$ and $y$ instead of $\tilde{x}$ and $\tilde{y}$.
System \eqref{sys case5 trans} in polar coordinates $x = r \cos \theta$, $y = r \sin \theta$ becomes system (d).
 \end{proof}

However  system  (e) from Theorem \ref{th-e} does not
have an isochronous center at the origin, since, generally speaking,  the
origin of the system is not a center, but a focus.
Indeed, system (e) can be written in the   Cartesian coordinates
 $x=r\cos(\theta), y=r\sin(\theta)$ as
\begin{equation}
\label{e2}
\begin{array}{l}
 \dot{x} = -y+k_1 x^2+k_2 x y+ x \left((k_3+k_4)x^2+(k_3-k_4)y^2+2 k_5xy \right),
\\
\dot{y} = x+k_1 xy+k_2 y^2+ y \left((k_3+k_4)x^2+(k_3-k_4)y^2+2 k_5xy \right).
\end{array}
\end{equation}
We computed the first two Lyapunov quantities for system \eqref{e2}
and  obtained    $\eta_1 =k_3 $ and $\eta_2 = 2k_1k_2k_5+k_4(k_1^2-k_2^2)$.
Thus,  the origin of system (e) is a focus, which is stable if
$k_3 < 0$ or $k_3 = 0, \eta_2<0$, and unstable if $k_3 > 0$ or $k_3 = 0, \eta_2>0$.
So, the necessary condition for existence of a center and a isochronous center at the origin of system (e) is $k_3 =\eta_2= 0$.

When $k_3 =\eta_2= 0$, by the linear transformation $x_1=x+\frac{k_2}{k_1} y$, $y_1=y-\frac{k_2}{k_1} x$, 
system \eqref{e2} is changed into
\begin{equation}
\label{e3}
\begin{array}{l}
 \dot{x} = -y+k_1 x^2-  \frac{k_1k_4}{k_2}x^2y,
\\
\dot{y} = x+k_1 xy- \frac{k_1k_4}{k_2}xy^2.
\end{array}
\end{equation}
System \eqref{e3} is a special case of system \eqref{case b-4 real} when $b_{20}=0$, which is system \eqref{system 1} under condition (4)
after adding the condition $b_{20}=0$. Therefore, when  $k_3 =\eta_2= 0$, the origin of system \eqref{e2}, and thus of system (e),
is an isochronous center.

It appears the authors of \cite{C-G-G} made the following
mistake in their reasoning.
They obtained system (e) from the condition
of vanishing of two period constants (period constants
are an analogue   of linearizability quantities when
computing in polar coordinates).
Then observing that the second equation of the
system is $\dot \theta=1$, they concluded that the system
has an isochronous center at the origin.
However, as we have shown, unless
 $k_3 =\eta_2= 0$, the origin of the system is an isochronous
 focus.


\bigskip


\section{Coexistence of isochronous centers}

In this section we present our study on existence of few isochronous centers  in real system \eqref{system 1}.

%
%

\begin{theorem} \label{th-coexist}
System \eqref{system 1} has at most two isochronous centers including  the origin when all parameters are real.
More precisely, under condition   \eqref{cond1*} and conditions (2), (3) and (4) of Theorem \ref{Theorem 1}, system \eqref{system 1} has at most two, one, two and two isochronous centers, respectively.
\end{theorem}

\begin{proof}

We first consider the simplest case, case (2) of Theorem \ref{Theorem 1}.
In this situation, system \eqref{system 1} has the form \eqref{case b-1 real}.
From the first equation of \eqref{case b-1 real}, we know that the coordinates of equilibria must satisfy
$b_{20} x+1=0$ or $a_{20} x^2-b_{20} x y-y=0$. Substituting $y=a_{20} x^2/(1+b_{20} x)$ into the right hand of the second equation of \eqref{case b-1 real}
we obtain $ 4 a_{20}^2 x^2+(b_{20}x+1)^2 =0$. Then, we get $x=0$ or $x=-1/b_{20}$.
On the other hand, substituting $x=-1/b_{20}$ into the right hand of the second equation of \eqref{case b-1 real},
we have $-3 a_{20} y/b_{20}=0$. Thus,  other than the origin $O:(0,0)$ we get the equilibrium $A: (-1/b_{20}, 0)$ when $a_{20} b_{20} \ne 0$,
no equilibria exist when $b_{20}=0$ and $a_{20} \ne 0$, or the line $x=-1/b_{20}$
is filled by equilibria when $a_{20} = 0$ and $b_{20} \ne 0$.

Computing the determinant of linear matrix for system \eqref{case b-1 real}  at the  equilibrium $A: (-1/b_{20}, 0)$,
we find that it is equal to $ -3 a_{20}^2/b_{20}^2<0$,
indicating that the equilibrium $A$ is a saddle if it exists. Clearly,  any equilibrium on  the line $x=-1/b_{20}$
cannot be an  isochronous center  when $a_{20} = 0$. Therefore, in the case (2) of Theorem \ref{Theorem 1},
system \eqref{system 1} has only one isochronous center at the origin.

\medskip

Consider now case (3) of Theorem \ref{Theorem 1}.
 In this case system \eqref{system 1} can be written as
 \begin{equation} \label{sys3-1}
\begin{aligned}
\dot{x}= &    (b_{20} x+1) (4 b_{11} x^2-b_{11} y^2-2 b_{20} x y-2 y)/2:=P_3(x,y),
\\
\dot{y}= &  x+b_{20} x^2+b_{11} x y-b_{20} y^2- (b_{11} b_{20}/2 ) y^3-b_{20}^2 x y^2+2 y b_{11} b_{20} x^2:=Q_3(x,y).
\end{aligned}
\end{equation}
From the first equation of \eqref{sys3-1} we see that the coordinates of equilibria must satisfy
$b_{20} x+1=0$ or $g_3(x,y):=4 b_{11} x^2-b_{11} y^2-2 b_{20} x y-2 y=0$.
Substituting $x=-1/b_{20}$ into the right hand side of the second equation of \eqref{sys3-1},
we have $- y b_{11} (b_{20}^2 y^2-2)/b_{20}=0$. Thus, we  find three  equilibria $A: (-1/b_{20}, 0)$
and $A_{\pm}: (-1/b_{20},  \pm  \sqrt{2}/b_{20} )$ if $b_{20} \ne 0$.

If we solve $g_3(x,y)=0$ and substitute the solution into the right hand side of the second equation of \eqref{sys3-1}
a very complicated expression arises. However, we only need to find the coordinates of centers of system \eqref{sys3-1}
and  for a singular point of the center type at  the trace of linear matrix is zero.
We calculate
 \begin{equation} \label{3TD}
\begin{aligned}
T_3(x,y):=  &     \frac{\partial P_3}{\partial x} +  \frac{\partial Q_3}{\partial y}
\\
  =&   b_{20} (4 b_{11} x^2-b_{11} y^2-2 b_{20} x y-2 y)/2
  + (b_{20} x+1) (8 b_{11} x-2 b_{20} y)/2
  \\& ~~+b_{11} x-2 b_{20} y-(3/2) b_{11} b_{20} y^2-2 b_{20}^2 x y+2 b_{11} b_{20} x^2,
\\
D_3(x,y):=  &     \frac{\partial P_3}{\partial x}   \frac{\partial Q_3}{\partial y}-\frac{\partial P_3}{\partial y}   \frac{\partial Q_3}{\partial x}
\\
= & ( b_{20} (4 b_{11} x^2-b_{11} y^2-2 b_{20} x y-2 y)/2+ (b_{20} x+1)) (8 b_{11} x-2 b_{20} y)/2)
 \\
 &~~ (b_{11} x-2 b_{20} y-(3/2) b_{11} b_{20} y^2-2 b_{20}^2 x y+2 b_{11} b_{20} x^2)
 \\
 & ~~-  (b_{20} x+1)  (-2 b_{11} y-2 b_{20} x-2) (4 b_{11} b_{20} x y-b_{20}^2 y^2+b_{11} y+2 b_{20} x+1)/2.
\end{aligned}
\end{equation}
Computing a Groebner basis of the ideal  $\langle g_3, Q_3, T_3 \rangle$
 we got  the basis
\[
\mathcal{G}_{3} := \{ b_{20} x^2+x, b_{11} y^2+2 b_{20} x y+2 y, b_{11} x  \}.
\]
When $b_{11}=0$ we obtain the  equilibrium
$O:(0,0)$  or the line $b_{20} x+1$ is filled with equilibria. When $b_{11}\ne0$, we obtain the equilibrium $B:(0, -2/b_{11})$.

Notice that all equilibria on the line $b_{20} x+1$ are degenerate when $b_{11}=0$, because the determinant of linear matrix
at each equilibrium is zero. Thus, an  equilibrium on the line $b_{20} x+1$ cannot be  isochronous centers when $b_{11}=0$.
By calculations, among all equilibria $A: (-1/b_{20}, 0)$, $A_{\pm}: (-1/b_{20},  \pm  \sqrt{2}/b_{20} )$  and $B:(0, -2/b_{11})$,
only at $B$ the trace  of linear part is zero and the determinant of linear part is positive at the same time.
So we only need check the isochronicity of equilibrium $B:(0, -2/b_{11})$.
Moving the equilibrium $B$ to the origin and making the change
\[
u=   \sqrt{2} (-2 b_{20}/b_{11})) x-\sqrt{2} y,~~~~v = \sqrt{2} x
\]
together with the time scaling $dt=-d\tau$,  we obtain from \eqref{sys3-1} the system
\begin{equation} \label{sys3-2}
\begin{aligned}
\dot{x}= &     -y - \frac{ \sqrt{2} b_{11}}{2} x y + \frac{\sqrt{2}b_{20}}{2} x^2 -\frac{ \sqrt{2} b_{20}}{2} y^2+ \frac{b_{11} b_{20}}{4} x^3 -
  x b_{11} b_{20} y^2 +  \frac{b_{20}^2}{2} x^2 y,
\\
\dot{y}= &   x + \frac{\sqrt{2} b_{11}}{4} x^2 + \sqrt{2} b_{20} x y -
  \sqrt{2} b_{11} y^2 + \frac{b_{11} b_{20}}{4} x^2 y + \frac{b_{20}^2}{2} x y^2 -
  b_{11} b_{20} y^3,
\end{aligned}
\end{equation}
where we still write $x, y$ instead of $u, v$.
It is easy to show that system \eqref{sys3-2} is Darboux  linearizable.
Therefore, the system has isochronous centers at the origin and at the point  $B:(0, -2/b_{11})$ if $ b_{11} \ne 0$.

\medskip
Now consider case (4)  of Theorem \ref{Theorem 1}.
In  this case 
  system \eqref{system 1} has the form
 \begin{equation} \label{sys4-1}
\begin{aligned}
\dot{x}= &    -y+a_{20} x^2-2 b_{20} x y+a_{20} b_{20} x^3+r_{11} x^2 y:=P_4(x,y),
\\
\dot{y}= &   x+a_{20} x y+b_{20} x^2-b_{20} y^2+a_{20} b_{20} x^2 y+r_{11} x y^2:=Q_4(x,y).
\end{aligned}
\end{equation}
It is difficult to find   the coordinates of equilibria of system \eqref{sys4-1} explicitly.
However,  we can calculate
 \begin{equation} \label{4TD}
\begin{aligned}
T_4(x,y):=  &     \frac{\partial P_4}{\partial x} +  \frac{\partial Q_4}{\partial y},
\\
D_4(x,y):=  &     \frac{\partial P_4}{\partial x}   \frac{\partial Q_4}{\partial y}-\frac{\partial P_4}{\partial y}   \frac{\partial Q_4}{\partial x}
\end{aligned}
\end{equation}
to find only coordinates of centers. Computing a Groebner basis of  $\langle P_4, Q_4, T_4 \rangle$
we obtained
 \begin{equation} \label{GB4-1}
\begin{aligned}
\mathcal{G}_{4}  :=  &  \{ a_{20} x y+4 b_{20} x^2+4 x, ~a_{20} y^2+4 b_{20} x y+4 y, ~-3 a_{20}^3 x+16 a_{20} b_{20}^2 x+16 a_{20} r_{11} x,
\\
  &    ~a_{20} x^2-4 b_{20} x y-4 y,  ~-3 a_{20}^2 y+16 b_{20}^2 y+16 r_{11} y, ~b_{20} x^3+b_{20} x y^2+x^2+y^2,
\\
  &~64 b_{20}^3 x^2+16 b_{20} r_{11} x^2-3 a_{20}^2 x-12 a_{20} b_{20} y+64 b_{20}^2 x+16 r_{11} x \}.
\end{aligned}
\end{equation}
Letting the first and the second polynomials in $\mathcal{G}_{4}$ be zeros, we get  $ y = -4 (b_{20} x+1)/a_{20} $ when $a_{20}\ne 0$
or $  x = y = 0 $. Substituting $y = -4 (b_{20} x+1)/a_{20}$ into $\mathcal{G}_{4}$, we have
\begin{equation} \label{GB4-2}
\begin{aligned}
  &  \{ 4(b_{20} x+1) (3 a_{20}^2-16 b_{20}^2-16 r_{11})/a_{20}, (16+(a_{20}^2+16 b_{20}^2) x^2+32 b_{20} x)/a_{20},
\\
  &     ~ (16+(a_{20}^2+16 b_{20}^2) x^2+32 b_{20} x) (b_{20} x+1)/a_{20}^2, -a_{20} x (3 a_{20}^2-16 b_{20}^2-16 r_{11}),
\\
  &~ (64 b_{20}^3+16 b_{20} r_{11}) x^2+(-3 a_{20}^2+112 b_{20}^2+16 r_{11}) x+48 b_{20} \}.
\end{aligned}
\end{equation}
Using the first  polynomial in \eqref{GB4-2}, we obtain   $ b_{20} x+1=0$, $y = 0 $ or
$ y = -4 (b_{20} x+1)/a_{20},  3 a_{20}^2-16 b_{20}^2-16 r_{11}=0 $. Substituting them in \eqref{GB4-2},
we obtain
\begin{eqnarray*}
&\{ (b_{20} x+1) x^2, ~a_{20} x^2, ~4 x (b_{20} x+1),  -x (-64 b_{20}^3 x-16 b_{20} r_{11} x+3 a_{20}^2-64 b_{20}^2-16 r_{11}),
 \\
&    -a_{20} x (3 a_{20}^2-16 b_{20}^2-16 r_{11})  \}
\end{eqnarray*}
and
\begin{eqnarray*}
&\{ (a_{20}^2 x^2 +(4 b_{20}x+4)^2)/a_{20}, ~(b_{20} x+1) (a_{20}^2 x^2+16 b_{20}^2 x^2+32 b_{20} x+16)/a_{20}^2,
 \\
&  3 b_{20} (a_{20}^2 x^2+16 b_{20}^2 x^2+32 b_{20} x+16) \},
\end{eqnarray*}
respectively. From the first and the second polynomials in above two sets, we see that on the line $y = -4 (b_{20} x+1)/a_{20}$ no center type equilibria exist when $a_{20}\ne 0$.

When $a_{20}= 0$, the basis $\mathcal{G}_{4}$ becomes $\{ b_{20} x^2+x, b_{20} x y+y, b_{20}^2 y+r_{11} y \}$, and we obtain that
 $ b_{20} x+1=0, (b_{20}^2 +r_{11}) y=0 $ or $ x = 0, y = 0$.
 When $a_{20}= 0$, $b_{20} x+1=0$ and $b_{20}^2 +r_{11}=0$, the line $b_{20} x+1=0$ is full of equilibria, none of which
 can be an isochronous center of system \eqref{sys4-1}. Hence, we only get the unique possible center  $A: (-1/b_{20}, 0)$
if $a_{20}= 0$, at which the trace of the linear matrix for system \eqref{sys4-1} is zero and the determinant is
$r_{11}/b_{20}^2+1$.
If $a_{20}= 0$, after moving the origin to the point $(-1/(2b_{20}), 0)$, system \eqref{sys4-1} is changed into
\begin{eqnarray}
\label{sys4-2}
\begin{aligned}
\dot{x}&= \frac{r_{11}}{4 b_{20}^2}y -\frac{2b_{20}^2+r_{11}}{b_{20}}xy +r_{11} x^2y,
\\
\dot{y} &= - \frac{1}{4b_{20}} +b_{20}x^2-\frac{ 2b_{20}^2+r_{11}}{2b_{20}} y^2 +r_{11} xy^2,
\end{aligned}
\end{eqnarray}
which is symmetric with respect to the $x$-axis. 
Moreover, equilibria $(\pm 1/(2b_{20}), 0)$ of system \eqref{sys4-2}
correspond to equilibria $A$ and $O$ of system \eqref{sys4-1}  respectively.
Thus,  except of the origin $O:(0,0)$ we get another isochronous center at the equilibrium  $A: (-1/b_{20}, 0)$  when $a_{20}= 0$, $r_{11}/b_{20}^2+1>0$ and
$ b_{20}\ne 0$.
Therefore, in case (4) system \eqref{system 1} has at most two isochronous centers.

\medskip


At last, we study the case when condition  \eqref{cond1*} is fulfilled.
In this situation let the vector field of system \eqref{system 1} be $(P_1(x,y), Q_1(x,y))$, as shown in \eqref{sys1-1}.
Similarly to  case (4), we only consider equilibria of center type avoiding complicated calculations of coordinates of all equilibria.
We calculate
 \begin{equation} \label{4TD}
\begin{aligned}
T_1(x,y):=  &     \frac{\partial P_1}{\partial x} +  \frac{\partial Q_1}{\partial y},
\\
D_1(x,y):=  &     \frac{\partial P_1}{\partial x}   \frac{\partial Q_1}{\partial y}-\frac{\partial P_1}{\partial y}   \frac{\partial Q_1}{\partial x}
\end{aligned}
\end{equation}
to find  coordinates of centers. 
The  Groebner basis of   $\langle P_1, Q_1, T_1 \rangle$ is
 \begin{equation} \label{GB1-1}
\begin{aligned}
\mathcal{G}_{1}  :=  &  \{ a_{20} y+b_{20} x+1, r_{11} x y+r_{20} x^2-r_{20} y^2+a_{20} x-b_{20} y, a_{20} r_{11} y^2+a_{20} r_{20} x y
\\
  &+b_{20} r_{20} y^2+a_{20}^2 y+b_{20}^2 y+r_{11} y+r_{20} x+a_{20} \}.
\end{aligned}
\end{equation}
If  $a_{20} =b_{20}=0$, system \eqref{sys1-1} cannot have other centers except of the origin. Without loss of generality
we suppose $b_{20} \ne 0$. If $a_{20} \ne 0$ the discussion is similar and we only need to make the change $(x,y) \to (y,x)$ with the time rescaling $dt=-d\tau$.
From the first  polynomial  in $\mathcal{G}_1$, we  get  $ x = -(a_{20} y+1)/b_{20} $.  Substituting it into $\mathcal{G}_1$, we have
\begin{equation} \label{GBg1}
g_1:= a_0+a_1 y+a_2 y^2=0,
\end{equation}
where $ a_0=a_{20} b_{20}-r_{20}$, $a_1=a_{20}^2 b_{20}+b_{20}^3-2 a_{20} r_{20}+b_{20} r_{11}$
and $a_2=-a_{20}^2 r_{20}+a_{20} b_{20} r_{11}+b_{20}^2 r_{20}$.
Thus, from \eqref{GBg1} we  find two roots $y_{\pm}=  (-a_1\pm \sqrt{a_1^2-4a_2a_0})/(2a_2)$ and then get two equilibria
$C_{\pm}: (-(a_{20} y_{\pm}+1)/b_{20},  y_{\pm})$ when $d_0:=a_1^2-4a_2a_0>0$ and $a_2\ne0$.
At $C_{\pm}$ the trace of linear matrix for system \eqref{sys1-1} is zero and the determinant of that is
\begin{eqnarray*}
\tilde{D}_{\pm}:= \frac{ d_0(\mp(a_{20}^2+b_{20}^2) \sqrt{d_0} -b_{20} (d_0/b_{20}^2-b_{20}^2 r_{11}+4 a_{20} b_{20} r_{20}+a_{20}^2 r_{11}-r_{11}^2-4 r_{20}^2))}
{2b_{20}^3(a_{20}^2 r_{20}-a_{20} b_{20} r_{11}-b_{20}^2 r_{20})^2}.
 \end{eqnarray*}
 Moreover,
 \[
\tilde{D}_+ \tilde{D}_-= -\frac{ d_0^2}{ b_{20}^4(a_{20}^2 r_{20}-a_{20} b_{20} r_{11}-b_{20}^2 r_{20})^2}<0,
 \]
 implying that at most one of $C_+$ and $C_-$ is a center.
 Actually, when  $r_{20} = a_{20} b_{20}$, we  find that the equilibrium $B:(0, -2/b_{11})$  is an isochronous center,
since it is easy to show that  system \eqref{sys1-1} is Darboux  linearizable at this point. 
Therefore, if  \eqref{cond1*} holds, then system \eqref{system 1} has at most two isochronous centers.
\end{proof}

To conclude, we have found conditions for isochronicity and linearizability of system \eqref{system 1} and clarified conditions
of isochronicity obtained in Chavarriga {\it et al} \cite{C-G-G}.
An important feature of our approach is the treatment of coefficients  of system \eqref{system 1} as
complex parameters, since this has allowed us to use formula \eqref{geincd2-IH} for finding
the decomposition of the integrability variety and to use the Radical Membership Test in order to
check the correctness computations involved modular arithmetics.

\section*{Acknowledgements}

The first author is partially supported by a CAPES grant.
The second author  is supported by the Slovenian Researcher Agency.
The fourth author is  supported by Marie Sklodowska-Curie Actions grant
655212-UBPDS-H2020-MSCA-IF-2014,  and partially supported by the National Natural
Science Foundations of China (No. 11431008).
The first and second authors   are also   supported by Marie Curie International Research Staff Exchange Scheme Fellowship
within the 7th European Community Framework Programme,
FP7-PEOPLE-2012-IRSES-316338.


\end{document}